\documentclass[12pt, reqno]{amsart}
     \makeatletter
     \def\section{\@startsection{section}{1}%
     \z@{.7\linespacing\@plus\linespacing}{.5\linespacing}%
     {\bfseries
     \centering
     }}
     \def\@secnumfont{\bfseries}
     \makeatother
\usepackage{amssymb}
\usepackage[final]{hyperref}
\usepackage{tikz}

\newtheorem{theorem}{Theorem}[section]
\newtheorem{prop}{Proposition}[section]
\newtheorem{lemma}[theorem]{Lemma}
\newtheorem*{prop*}{Proposition}

\newtheorem{definition}[theorem]{Definition}

\numberwithin{equation}{section}

\newcommand{\mbc}{{\mathbb C}}

\newcommand{\mbe}{{\mathbb E}}
\newcommand{\mbf}{{\mathbb F}}

\newcommand{\mbr}{{\mathbb R}}
\newcommand{\mbp}{{\mathbb P}}

\newcommand{\mbv}{{\mathbb V}}

\newcommand{\mca}{{\mathcal A}}
\newcommand{\mcb}{{\mathcal B}}
\newcommand{\mcf}{{\mathcal F}}
\newcommand{\mcg}{{\mathcal G}}

\newcommand{\mco}{{\mathcal O}}

\newcommand{\mcfh}{\mcf(H)}

\newcommand{\Tr}{{\rm Tr}}
\newcommand{\tr}{{\rm tr}}

\newcommand{\la}{{\langle}}

\newcommand{\ra}{{\rangle}}

\newcommand{\mc}[1]{\mathcal{#1}}
\newcommand{\norm}[1]{\left\Vert#1\right\Vert}
\newcommand{\chf}[1]{\mathbf{1}_{#1}}
\newcommand{\set}[1]{\left\{#1\right\}}

\begin{document}

\title[Quantum Free Yang-Mills on the Plane]{Quantum Free Yang-Mills on the Plane}


\author[M. Anshelevich] {Michael Anshelevich}
\address{Department of Mathematics, Texas A\&M University,
College Station, TX 77843}
\email{manshel@math.tamu.edu }
\urladdr{http://www.math.tamu.edu/$\sim$manshel/}
\thanks{Research supported by  US National Science Foundation Grant DMS-0900935.}

\author[A. N. Sengupta] {Ambar N.\ Sengupta}
\address{\hskip -.05in Department of Mathematics, Louisiana State University,
Baton Rouge, LA 70803}
\email{sengupta@math.lsu.edu}
\urladdr{http://www.math.lsu.edu/$\sim$sengupta}
\thanks{Research  supported by   US National Science Foundation  Grant DMS-0601141. }

\date{June 2, 2011}


\begin{abstract} We construct a free-probability quantum Yang-Mills theory on the two dimensional plane,
determine the Wilson loop expectation values, and show that this theory is the $N=\infty$ limit of
$U(N)$ quantum Yang-Mills theory on the plane.
\end{abstract}
\maketitle

\section{Introduction}\label{s:intr}

 In this paper we use free probability and free stochastic calculus to construct a large-$N$ limit of $U(N)$ quantum Yang-Mills theory on the Euclidean plane $\mbr^2$.  While it has long been   expected that free probability should play a central role in describing large-$N$ limits of certain matrix-model quantum   theories  (see, for instance, Douglas \cite{Do97} and Gopakumar and Gross \cite{GG95}), the model we develop, along with the earlier work of Xu \cite{Xu97}   in this context, may be the first concrete rigorously developed example of a free-probability based geometric quantum field theory.

 Pure quantum Yang-Mills theory, with gauge group $U(N)$ and spacetime    $\mbr^2$, is described  by the Yang-Mills measure $\mu_g$ which is formally
 a measure   on gauge equivalence classes of connections $A$ and has formal  density
$e^{-\frac{1}{2g^2}|\!|F^A|\!|^2_{L^2}}$, where $F^A$ is the curvature of a connection form $A$ and $g$ a coupling constant. It has  been known,
at least since the work of `t Hooft \cite{tHft74}, that this theory has a meaningful and conceptually useful limit as $N\to\infty$, holding
$g^2N$ fixed.

There is a vast body of works in the physics literature on the large-$N$ limit of $U(N)$ gauge theories
(early works include those of Kazakov and Kostov \cite{KK80, KK81}).
On the mathematical side, Singer  \cite{Si95} showed that a mathematically
meaningful   `master field'  exists as the large-$N$ limit of two-dimensional $U(N)$ quantum Yang-Mills theory. Our free-probability Yang-Mills
theory may be viewed as a realization of this master field theory.  The large-$N$ limit of  Wilson loop expectations in quantum Yang-Mills on $\mbr^2$
was studied
by Xu \cite{Xu97} using the known distributions of the $U(N)$ Wilson loop expectation values. For a brief review see \cite{Sen08a}.

\section{Classical, Quantum, and Free}\label{s:fym}

In making comparisons  of our theory to classical differential geometric gauge theory, one should have in mind a $U(N)$ principal bundle over $\mbr^2$.
For convenience we may trivialize the bundle and take the space $\mathcal A$ of all  connections as the space of smooth $u(N)$-valued $1$-forms on $\mbr^2$, where $u(N)$ is the Lie algebra
of skew-hermitian $N\times N$ matrices.  The group $\mcg$ of gauge transformations consists of smooth maps $\phi:\mbr^2\to U(N)$ and acts on $\mca$ by
$$A^{\phi}={\phi}A{\phi}^{-1}+\phi d\phi^{-1}.$$
It is usually
more convenient to work with ${\mathcal G}_o$, the subgroup of ${\mathcal G}$  consisting of transformations which are the identity over  the basepoint
$o=(0,0)$.  Parallel transport  by $A$ along a piecewise smooth path
\[
c:[T_0,T_1]\to\mbr^2
\]
is given by $h_c(T_1)$ where $h_c:[T_0,T_1]\to U(N)$ solves
$$dh_c(t)=-A\bigl(c'(t)\bigr)h_c(t)\,dt\qquad\hbox{and}\qquad h_c(T_0)=I,$$
with $t$ running over $[T_0,T_1]$.  For any $A\in\mca$ there is a $\phi\in \mcg_o$ for which $A_{\rm rad}=A^\phi$ vanishes on radial vectors; this is
radial gauge fixing, and identifies $\mca/\mcg_o$ with the linear space $\mca_o$   of $u(N)$-valued functions   on $\mbr^2$ by associating
each $A$ to $f^{A_{\rm rad}}$ times the area $2$-form on $\mbr^2$. The equation of parallel transport then associates to each $f\in\mca_o$
and path $c:[T_0,T_1]\to\mbr^2:t\mapsto r_c(t)e^{it}$, the differential equation
\begin{equation}\label{e:dhcc}
dh_c(t)=i dM^f_c(t)h_c(t)\,\qquad\hbox{and}\qquad h_c(T_0)=I,\end{equation}
where $-iM^f_c(t)$ is the integral of $f$ times area-form over the cone
\begin{equation}\label{E:defSct}
S_c(t)=\{re^{i{\theta}}\,:\, \theta\in [T_0,t],\, 0\leq r\leq r_c({\theta})\}.\end{equation}
Primarily we are interested in loops $c$ based at the origin $o$, and then $h_c(T_1)$ is the classical {\em holonomy} of the connection on the loop $c$.

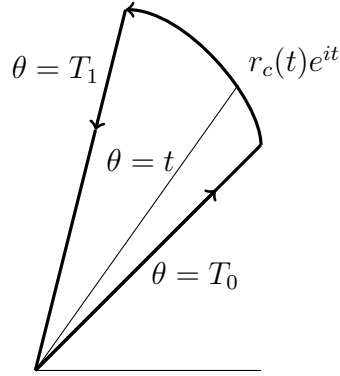
\begin{figure}
\begin{center}
 \begin{tikzpicture}[scale=2]
\draw (0,0) -- (1.5,0);

\draw [->,very thick](0,0) -- (1.2,1.2);

\draw[very thick]  (0,0) -- (1.5,1.5);

\draw [<-, very thick](0.4,1.6) -- (0.6, 2.4);

\draw [very thick](0.4,1.6) -- (0, 0);

 \draw  [->,very thick] (1.5,1.5) .. controls (1.5,1.8) and (1,2.4) .. (0.6,2.4);

 \draw (0,0)--(1.35,1.9);

\coordinate [label=left:$r_c(t)e^{it}$] (cout) at (2.1, 2.05);

\coordinate [label=left: ${\theta=t}$] (t) at (1,1.4);

\coordinate [label=right: ${\theta=T_0}$] (t0) at (.7,.63);

\coordinate [label=left: ${\theta=T_1}$] (t1) at (.5,2);

\end{tikzpicture}
\end{center}
\caption{A cross radial path }
\label{fig:crpcone}
\end{figure}

For quantum Yang-Mills theory with gauge group $U(N)$ the Yang-Mills measure is a probability measure specified formally by the expression
\begin{equation*}
d\mu_g(A)=\frac{1}{Z_g}e^{-\frac{1}{2g^2}|\!|F^A|\!|^2_{L^2}}\, [DA] \end{equation*}
where $F^A=dA+A\wedge A$ is the curvature of a connection form $A$, $g$ is a  parameter that may be viewed as a `coupling constant',
and $[DA]$ is formal Lebesgue measure on $\mca_o$.  Then $\mu_g$ is realized rigorously
as Gaussian measure
\begin{equation}\label{e:formalymN}
d\mu_g(f)=\frac{1}{Z_g}e^{-\frac{1}{2g^2}|\!|f|\!|^2_{L^2}}\, Df \end{equation}
on a suitable completion of $L^2(\mbr^2)\otimes u(N)$.  This will be reviewed in more detail below in section \ref{s:revN}, but for the sake of motivation for the free theory let us take a very quick look at the framework of the $U(N)$ theory. To each $f\in L^2(\mbr^2)$ is associated a $u(N)$-valued random variable
$ib_{N,f}$, where $b_{N,f}$ is an $N\times N$ Gaussian hermitian random matrix with mean $0$ and entries having variances determined
 by $\frac{1}{N}|\!|f|\!|^2_{L^2}$.  Following an idea of L.\ Gross \cite{GKS89} the classical differential
equation (\ref{e:dhcc}) is replaced by the It\^o  stochastic differential equation
\begin{equation}\label{e:dhcs}
dh_c(t)=  i dM_{N,c}(t)h_c(t)-\frac{1}{2}dM_{N,c}(t)^2h_c(t)\,\qquad\hbox{and}\qquad h_c(T_0)=I,\end{equation}
where now $M_{N,c}(t)$ is $b_{N,f}$ with $f=1_{S_c(t)}$. These equations were studied and the full theory of the quantum Yang-Mills measure on the plane
were developed by Gross et al. \cite{GKS89} and Driver \cite{Dr89}.

The large-$N$ limit of the white noise process $f\mapsto b_{N,f}$, as $N\to\infty$, can be described by using free probability \cite{Sen08}.
We proceed now to outline this framework very briefly and also describe the corresponding parallel transport process.

Let $H$ be the    Hilbert space $L^2(\mbr^2)$, and for $f\in H$ let
\begin{equation}\label{e:defbf}
b_f=a_f+c_f,\end{equation}
where $a_f=c_f^*$ and $c_f$ is the creation operator
\begin{equation}\label{e:defcf}c_f:{\mathcal F}(H)\to {\mathcal F}(H): x\mapsto f\otimes x \end{equation}
with $\mcf(H)$ being the full Fock space
\begin{equation}\label{e:defFH}
{\mathcal F}(H)=\bigoplus_{n= 0}^{\infty}H^{\otimes n}.\end{equation}
 We work in the closed subalgebra   $\mcb_{\mbr^2}$
of the $C^*$-algebra of  bounded linear operators on $\mcf(H)$ generated by all $b_f$, with the `expectation' operator $\tau(T)=\la T \Omega,\Omega\ra$, where $\Omega=1\in \mbc\subset\mcf(H)$.


For any Borel $S\subset\mbr^2$ there is the closed subalgebra $\mcb_S$ generated by $b_f$, with  $f$ zero off $S$. If $S_1,\ldots, S_m$ are mutually disjoint Borel subsets of $\mbr^2$, then $\mcb_{S_1},\ldots,\mcb_{S_m}$ are freely independent.  We view $f\mapsto b_f$ as a `free noise' process.
(It  makes sense  even if $\mbr^2$ is replaced by any measure space.)

Consider a  continuous  path
$$c:[T_0,T_1]\to\mbr^2 = \mbc:t\mapsto r_c(t) e^{it}$$
with $r_c\geq 0$.
Let
\begin{equation}\label{E:defMc}
M_c(t)=b_{1_{S_c(t)}},\end{equation}
where $S_c(t)$ is, as before, the cone:
\begin{equation}\label{E:defSct2}
S_c(t)=\{re^{i{\theta}}: \theta\in [T_0,t], 0\leq r\leq r_c(\theta)\}.\end{equation}
The {\em parallel-transport} process along the path $c$ is a path $t\mapsto u_c(t)\in B\bigl(\mcf(H)\bigr)$
solving the free stochastic differential equation
\begin{equation}\label{UtMt}du_c = i \bigl(dM_c \bigr) u_c    -\frac{1}{2}\bigl(dM_c \bigr)^2u_c,\qquad\hbox{with}\quad u_c(T_0)=I.\end{equation} We will study such equations below in section \ref{s:fs}.
 (In terms of the notation used later in (\ref{E:ux}), $u_c$ is $u_{M_c}$.) To avoid excessive notation we indulge in a small amount of notational ambiguity by denoting the parallel transport along the full path $c$ also by $u_c$:
\begin{equation}\label{E:defucT1}u_c=u_c(T_1).\end{equation}

 We identify paths which are reparametrizations of each other by `positive speed' (hence strictly increasing) smooth functions.
If $c:[0,1]\to\mbr^2$ is a path then its {\em reverse} is  $c^{-1}:[0,1]\to\mbr^2: t\mapsto c(1-t)$.  Following L\'evy \cite{Le10} we denote the initial point $c(0)$ by $\underline{c}$ and the final point $c(1)$ by $\overline{c}$.

A    path  of the form $[T_0,T_1]\to \mbc: t\mapsto r_c(t)e^{it}$ is {\em cross radial}. A  path of the form $[r_0,r_1]\to\mbr^2: r\mapsto re^{it}$, for some fixed $t\in \mbr$, and $0\leq r_0\leq r_1$, is {\em radial}.

We have already defined $u_c$ for cross radial $c$; for such $c$ we also define
\begin{equation}\label{e:defucbar}
u_{c^{-1}}=u_c^{-1}.\end{equation}

By a  {\em basic path} in $\mbr^2$ we shall mean a path which is the composite of radial paths, cross radial arcs, and their reverses, wth different paths intersecting at most at endpoints.  We do not distinguish between paths that are    reparametrizations of each other by smooth reparametrization functions of strictly positive derivative.  Moreover, by following a path as long as it is radial (or reverse) and then as long as it is cross radial (or reverse), and so on, every basic path is the composite of a unique sequence of radial/cross radial paths and their reverses.

The initial point of a path $c$  is denoted by $\underline{c}$, and the final point by $\overline{c}$.

Next let us define the notion of {\em backtrack equivalence}, following L\'evy \cite{Le10} who treats a more general situation.  Basic paths $c_1$ and $c_2$  are  {\em elementarily  equivalent}
  if there are basic paths $a$, $b$, $d$ such that
$$\{c_1, c_2\}=\{add^{-1}b, ab\}.$$
Thus, one of the $c_i$ is obtained from the other by erasing the backtracking part $dd^{-1}$.   Basic paths $c$ and $c'$ are {\em backtrack equivalent}, denoted
$$c\simeq_{\rm bt} c',$$
 if there is a sequence of basic paths $c=c_0, c_1,\ldots, c_n=c'$ where $c_i$ is elementarily equivalent to $c_{i+1}$  for $i\in\{0,\ldots, n-1\}$.

  Clearly, backtrack equivalence is in fact an equivalence relation.  Moreover,
  if
  $$a\simeq_{\rm bt} a'\quad\hbox{and}\quad b\simeq_{\rm bt} b'$$
  then
\begin{equation}\label{E:abapbp}
ab\simeq_{\rm bt} a'b',\end{equation}
  if the composites $ab$ and $a'b'$ exist.

\begin{definition}\label{d:hol}  For a basic curve $c:[T_0,T_1]\to\mbr^2$ we define
the {\em free parallel transport} along $c$ to be
\begin{equation}\label{E:defhol}
u_c=u_{c_m}\cdots u_{c_1}\end{equation}
if $c$ is the composite
$$c=c_m\cdots c_1,$$
where each $c_j$ or its reverse is radial or cross radial, and $u_{c_j}=1$ if $c_j$ is radial.  If $c$ is a loop, then we call $u_c$ the {\em free holonomy} around $c$.  \end{definition}

The definition of backtrack equivalence makes it clear that {\em holonomy}, in all its forms, {\em ignores backtracks}, in the sense that backtrack equivalent loops have the same holonomy.

\section{Free Stochastics}\label{s:fs}

In this section we introduce notions and prove results in free stochastic calculus that are motivated by and useful for our study of stochastic holonomy in the next sections.

Let $H$ be a   separable   Hilbert space and
$${\mathcal F}(H)=\bigoplus_{n\geq 0}H^{\otimes n}$$
the full Fock space over $H$ (if we start with a real Hilbert space we should complexify the Fock space). For any bounded linear operator $T$ on $\mcf(H)$ let
$$\tau(T)=\la T\Omega,\Omega\ra$$
where $\Omega=1\in\mbc\subset\mcf(H)$. For $f\in H$, let $c_f$ and $a_f$ be, respectively,
the creation and annihilation operators on ${\mathcal F}(H)$ specified by
\begin{equation}\label{E:defcfaf}\begin{split}
c_f(v_1\otimes\cdots\otimes v_n) &=f\otimes f_1\otimes\cdots\otimes f_n\\
 a_f(f_1\otimes\cdots\otimes f_n)&=\la f,f_1\ra
f_2\otimes\cdots\otimes f_n,\end{split}\end{equation}
with $a_f$ being $0$ on $  \mbc$. Let
\begin{equation}\label{E:defbf2}
b_f=a_f+c_f.\end{equation}

\begin{prop}\label{P:} Let $\mco$ be the norm-closed $C^*$-algebra of bounded operators on $\mcfh$ generated by the elements $b_f$ with $f$ running over
$H$. Then $\tau$ is faithful and tracial on $\mco$ and $\Omega=1\in\mbc$ is a cyclic vector, that is, $\{b_f\Omega\,|\,f\in H\}$ is a dense subspace
of $\mcfh$.

\end{prop}
\begin{proof}
See Section 2.6 of \cite{VDN}.
\end{proof}

We will work henceforth with the algebraic probability space $(\mco,\tau)$.

By a {\em free Brownian motion} $x$ on $(\mco,\tau)$ we shall mean a path
$x:[T_0,T_1]\to\mco$, for some $T_0, T_1\in (0,\infty)$ with $T_0<T_1$, for which
\begin{itemize}
\item[(i)] each $x(t)$ is real (hermitian), and $x(T_0)=0$;
\item[(ii)] if $t_0\leq t_1 \leq \ldots \leq t_n$ in $[T_0,T_1]$ then
\[
x(t_0), \ x(t_1)-x(t_0), \ldots, x(t_n) - x(t_{n-1})
\]
are freely independent;
\item[(iii)] there is a strictly increasing continuous function $A_x:[T_0,T_1]\to [0,\infty)$,
with $A_x(T_0)=0$,  such that if $0\leq t_0<t_1$ then $x(t_1)-x(t_0)$ has centered semicircular distribution with variance $A_x(t_1)-A_x(t_0)$.
\end{itemize}

The function $A_x(\cdot)$ will be called the {\em clock} of the free Brownian motion. If $A_x(t)=t$ for all $t\in [T_0,T_1]$, then the process is the {\em standard} free
Brownian motion on $(\mco,\tau)$.  The clock can be recovered from the free Brownian motion from
\begin{equation}\label{E:clockform}
(dx)^2=dA_x, \end{equation}
where, on the left we have the square of a free stochastic differential.


\begin{lemma}\label{L:sumbrown}
Suppose $x$ and $z$ are freely independent free Brownian motions, i.e.,   the algebra  generated by $x([T_0,T_1])$ is freely independent of the algebra generated by   $z([T_0,T_1])$.
Then $x+z$ is also a free Brownian motion, with the clock being the sum of the clocks for $x$ and $z$. \end{lemma}
We omit the proof, which is straightforward.

We will make much use of a multiplicative analog of free brownian motion.  First consider a free   Brownian process
\begin{equation}\label{E:freebm}
[T_0,T_1]\to\mco:t\mapsto b(t)\qquad\hbox{with $b(0)=0$,}\end{equation}
clocked by $t-T_0$.
The  corresponding {\em free multiplicative brownian motion} is a process $[T_0,T_1]\to \mco:t\mapsto v(t)$ satisfying the free stochastic differential equation
\begin{equation}\label{E:freembm}
du_b(t)=idt\,u_b(t)-\frac{1}{2}u_b(t)dt,\qquad\hbox{with $u_b(T_0)=I$.}\end{equation}
 That this equation has a unique solution follows by the results of K\"um\-mer\-er and Speicher  \cite[section 5]{KS92}.   Biane \cite[Theorem 2]{BiaMultiplicativeBM} (see also     \cite[Lemma 6.3]{BV92}) shows that
 the moments of $u_b(t)$ are given by
 \begin{equation}\label{E:momvt}
 \tau\bigl(u_b(t)^k\bigr) =e^{- \frac{k}{2}t}P_k(t),
\end{equation}
for all non-negative integers $k$,
where   $P_k(t)$ is the polynomial of degree $k-1$ in $t$ that, for $k\geq 1$, satisfies
\begin{equation}\label{diffeqPk}
\frac{dP_k(t)}{dt}=-\frac{k}{2}\sum_{j=1}^{k-1}P_j(t)P_{k-j}(t),
\end{equation}
with $P_k(0)=1$, and $P_1(t)=P_0(t)=1$. (The distribution of $u_b(t)^\ast$ is the same as that $u_b(t)$.)    An explicit formula for $P_k(t)$ is given by Biane \cite[Remark on p. 267]{Bia97b}\cite[Lemma 1]{BiaMultiplicativeBM} and Singer \cite{Si95}\begin{equation}\label{E:Pkt}
P_k(t)=\sum_{k=0}^{n-1}(-1)^k\frac{t^k}{k!}n^{k-1}\binom{n}{k+1}.\end{equation}
 and can be expressed using an associated Laguerre polynomial $ \frac{1}{n} L_{n-1}^{(1)}(nt)$.

 We restate this and more as:

\begin{prop}\label{P:unitbr}  Suppose $x:[T_0,T_1] \rightarrow \mco$ is a free Brownian motion in $(\mco, \tau)$. Then there is a unique mapping $[T_0,T_1]\to \mco:t\mapsto u_x(t)$, with $u_x(T_0)=I$, satisfying
\begin{equation}\label{E:ux}
du_x  = i \,dx\, u_x -\frac{1}{2} u_x \,dA_x \end{equation}
The element $u(t)$ is unitary and belongs to the $C^*$-algebra generated by $x([T_0,t])$, for all $t\in [T_0,T_1]$.
Moreover, for $T_0\leq s\leq t<T_1$, the multiplicative increments $u_x(t) u_x(s)^\ast$ are freely independent and have stationary distributions, depending only of $A_x(t)-A_x(s)$. Specifically, the distribution of $u(t)u(s)^*$ is    multiplicative semicircular with parameter $A_x(t)-A_x(s)$, specified  explicitly by
\begin{equation}
\label{E:distribux}
\tau\bigl((u(t)u(s)^*)^k\bigr)= P_k\bigl(n [A_x(t)-A_x(s)]\bigr)e^{-k[A_x(t)-A_x(s)]/2}\end{equation}
for all non-negative integers $k$.
\end{prop}

We   call the process $t\mapsto u_x(t)$ the {\em free multiplicative Brownian motion generated by} $x$.
Looking back at (\ref{UtMt}) we see that $u_c$ is the same as $u_{M_c}$, in this notation.

\begin{proof} The process ${\tilde x}: s\mapsto  x\bigl(A_x^{-1}(s)\bigr)$ is free   Brownian motion with clock $s$. Then the result follows
  from Theorem 2 of \cite{BiaMultiplicativeBM} applied to ${\tilde x}$. \end{proof}

\begin{lemma}
\label{L:invar}
Suppose $z, u\in \mco$, with $u$ unitary lying in a $\ast$-closed algebra $\mcb\subset\mco$ which is freely independent of $z$. Then
$u^\ast z u$ is freely independent of $\mcb$ and has the same distribution as $z$. 
\end{lemma}


\begin{proof} The argument is standard.  Since $\tau$ is a trace, and $u$ is unitary,
\[
\tau \Bigl( (u^\ast z u)^n \Bigr) = \tau(z^n);
\]
thus, $u^\ast z u$ has the same distribution as $z$. Similarly, let $P_1, P_2, \ldots, P_n$ be polynomials such that $\tau(P_i(u^\ast z u)) = \tau(P_i(z)) = 0$. Then for any $b_1, \ldots, b_n \in \mcb$ with $\tau(b_i) = 0$, we have
\[
\begin{split}
& \tau \Bigl( P_1(u^\ast z u) b_1 P_2(u^\ast z u) b_2 \ldots P_n(u^\ast z u) b_n \Bigr) \\
&\quad = \tau \Bigl( P_1(z) (u b_1 u^\ast) P_2(z) (u b_2 u^\ast) \ldots P_n(z) (u b_n u^\ast) \Bigr) = 0
\end{split}
\]
since $z$ is freely independent from $\mcb$ and $\tau(u b u^\ast) = \tau(b)$. It follows that $u^\ast z u$ is freely independent from $\mcb$ as well.
\end{proof}

\begin{prop}\label{P:invar} Suppose $z: [T_0,T_1]\to\mco$ is  a free Brownian motion, $u:[T_0,T_1]\to\mco$ a map,
and $\mcb$ is a   subalgebra of $\mco$ such that each $u(t)$  is a unitary and belongs to $\mcb$, and the algebra generated by
$z([T_0,T_1])$ is freely independent of $\mcb$. Then the   equation
\begin{equation}
\label{E:vuz}
dz^u(t)= u(t)^\ast \,dz(t)\, u(t)
\end{equation}
has a unique solution satisfying $z^u(T_0)=I$, and $z^u$ is a free Brownian motion with the same clock as $z$. Moreover,   $\{z^u(t)\,:\, t\in [T_0,T_1]\} $ is freely independent of $\mcb$.
\end{prop}

\begin{proof} The existence of a unique solution follows from the general results of  \cite{KS92} (see also Section~4 of \cite{Bia97b}). Alternatively, we can write \cite{BiaSpeBrownian}:
\[
z^u(t) = \int_{T_0}^t (u(t)^\ast \otimes u(t)) \sharp \,dz(t).
\]
Observe that
\[
\norm{(u^\ast \otimes u) \chf{[T_0,t)}}_{\mcb_\infty} = (t-T_0)^{1/2} < \infty,
\]
and the biprocess is clearly adapted, so the integral is well-defined. Now choose a partition
\[
T_0 = t_0 < t_1 < \ldots < t_n = t.
\]
Denote $\mco_i = C^\ast(\mc{B}, z(t_1), \ldots, z(t_i))$. Then Lemma~\ref{L:invar} says that for each $i$,
\[
u(t_{i-1})^\ast (z(t_i) - z(t_{i-1})) u(t_{i-1})
\]
is freely independent from $\mco_{i-1}$ and has the same distribution as $z(t_i) - z(t_{i-1})$. Therefore
the $n$-tuple of elements

\[
\left(u(t_{k-1})^\ast (z(t_k) - z(t_{k-1})) u(t_{k-1})\right)_{k=1,...,n}
\]
is freely independent from $\mcb$ and has the same joint distribution as

\[
\left(  (z(t_k) - z(t_{k-1})) \right)_{k=1,...,n},
\]

namely they are free semicircular with variances
\[
\left(A_z(t_k) - A_z(t_{k-1})\right)_{k=1,...,n}.
\]
Since $\int_{T_0}^t (u(t)^\ast \otimes u(t)) \sharp \,dz(t)$ is the limit, in the operator norm, of the  elements
\[
\sum_{i=1}^n u(t_{i-1})^\ast (z(t_i) - z(t_{i-1})) u(t_{i-1}),
\]

all the conclusions follow.
\end{proof}

In fact, the arguments show the following stronger result:

\begin{prop}
Let $\set{\mco_t: t \in [T_0, T_1]}$ be a filtration. Suppose $z: [T_0,T_1] \to \mco$ and $u: [T_0,T_1] \to \mco$ are processes adapted to this filtration, that is, $z(t), u(t) \in \mco_t$ for all $t$. Suppose in addition that each $u(t)$ is unitary, and $z$ is a free Brownian motion adapted to the filtration in the stronger sense that for $t > s$, $z(t) - z(s)$ is freely independent from $\mco_s$. Then the equation
\begin{equation}
dz^u(t) = u(t)^\ast dz(t) u(t)
\end{equation}
has a unique solution satisfying $z^u(T_0) = I$, and $z^u$ is a free Brownian motion with the same clock as $z$.
\end{prop}

Now we turn to the main result needed:

\begin{prop}
\label{P:cocycle}
Let $x$ and $z$ be freely independent free Brownian motions in $(\mco, \tau)$ with time parameter running over $[T_0,T_1]$. Then:
\begin{equation}
\label{E:uxz}
u_x^\ast u_{x+z} = u_{y}
\end{equation}
where         $y=z^{u_x}$.

Moreover, if $\mcb$ is a $\ast$-closed subalgebra of $\mco$ such that $x([T_0,T_1])\subset\mcb$, and $\mcb$ is freely independent
of  $z([T_0,T_1])$, then  $u_y([T_0,T_1])$ is freely independent of $\mcb$.
\end{prop}

By Proposition \ref{P:unitbr}, $u_y(t)$ is multiplicative semicircular with parameter  $A_z(t)$ (see, in this context, the remark involving (\ref{E:distribux}) with $A_z(t)$ in place of $A_x(t)$).

\begin{proof}
Denote $u(t) = u_x(t)^\ast u_{x+z}(t)$. Recall that
\[
d u_{x+z} = i \,(dx + dz)\, u_{x+z} - \frac{1}{2} u_{x+z} \,dA_{x+z}
\]
and
\[
d u_x^\ast = - i u_x^\ast \,dx  - \frac{1}{2} u_x^\ast \,dA_x.
\]
Then using the free It{\^o} product formula,
\[
\begin{split}
d u & = u_x^\ast (d u_{x+z}) + (d u_x^\ast) u_{x+z} + (d u_x^\ast) (d u_{x+z}) \\
& = i u_x^\ast \,(dx + dz)\, u_{x+z} - \frac{1}{2} u_x^\ast u_{x+z} \,dA_{x+z} - i u_x^\ast \,dx\, u_{x+z} \\
&\qquad - \frac{1}{2} u_x^\ast u_{x+z} \,dA_x
  + u_x^\ast \,dx (dx + dz)\, u_{x+z} \\
& = i u_x^\ast \,dz\, u_{x+z} - \frac{1}{2} u_x^\ast u_{x+z} (dA_{x+z} + dA_x) + u_x^\ast u_{x+z} \,dA_x \\
& = i u_x^\ast \,dz\, u_{x+z} - \frac{1}{2} u_x^\ast u_{x+z} \,dA_{z} \\
& = i u_{x}^\ast \,dz\, u_x u - \frac{1}{2} u \,dA_{z} \\
& = i \,dz^{u_x} u - \frac{1}{2} u \,dA_{z^{u_x}}.\qquad\hbox{by (\ref{E:vuz}).}
\end{split}
\]
Therefore
\[
u_x(t)^\ast u_{x+z}(t) = u_{z^{u_x}}(t).
\]
By Proposition~\ref{P:unitbr}, if $x([T_0,T_1])\subset\mcb$ then also $u_x([T_0, T_1]) \subset \mcb$. Therefore by Proposition~\ref{P:invar}, $ {y}([T_0, T_1])$ is freely independent from $\mcb$, where $y=z^{u_x}$. Again, by Proposition ~\ref{P:unitbr}, $ u_{y}([T_0, T_1])$ is contained inside   the closed algebra generated by $ {y}([T_0, T_1])$, and so is freely independent of $\mcb$.
\end{proof}

\section{Lasso Holonomies}\label{s:lh}

In this section we  focus on a special class of loops, which we call standard lassos, that we will use later to study holonomies for more
general loops.

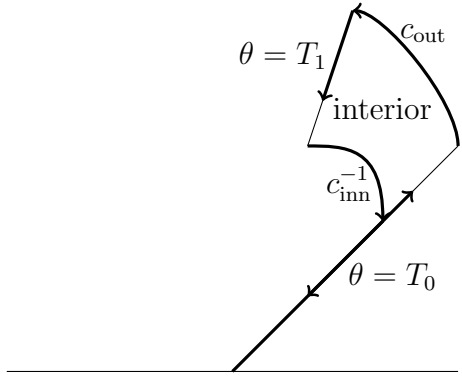
\begin{figure}
\begin{center}
 \begin{tikzpicture}[scale=2]
\draw (-1.5,0) -- (1.5,0);

\draw [->,very thick](0,0) -- (1.2,1.2);

\draw  (0,0) -- (1.5,1.5);
\draw [<-,very thick] (1,1) .. controls (1,1.5) and (0.75,1.5) .. (0.5,1.5);

\draw [<-, very thick](0.6,1.8) -- (0.8, 2.4);
\draw  (0.5,1.5) -- (0.8, 2.4);
\draw  [->,very thick] (1.5,1.5) .. controls (1.5,1.8) and (1,2.4) .. (0.8,2.4);

\draw [->,very thick](1,1) -- (0.5,0.5);

\coordinate [label=left:$ {c}_{\rm inn}^{-1}$] (cinn) at (1,1.25);

\coordinate [label=left:$ {c}_{\rm out}$] (cout) at (1.5,2.25);

\coordinate [label=left: interior] (cout) at (1.4,1.75);

\coordinate [label=right: ${\theta=T_0}$] (t0) at (.7,.63);

\coordinate [label=left: ${\theta=T_1}$] (t1) at (.7,2.1);

\end{tikzpicture}
\end{center}
\caption{A standard lasso}\label{fig:sl}
\end{figure}

By a {\em  standard lasso}  (see Figure \ref{fig:sl}) we mean a loop in $\mbr^2$ which starts
at the origin $o$, travels out along a radial path  $[0,r_{1}(T_0)]\to\mbc: r\mapsto re^{i T_0}$ for some fixed $T_0\in [0,2\pi)$ and $r_1(T_0)\geq 0$, followed by a cross radial arc $[T_0, T_1]\to\mbr^2:t\mapsto c_{\rm out}(t)=r_1(t)e^{it}$ for some $T_1\in (T_0,2\pi]$,
followed by the reverse of a radial path $[r_{0}(T_1), r_{1}(T_1)]\to\mbr^2:r\mapsto re^{i T_1}$ for some $r_0(T_1)\in [0,r_1(T_1)]$,
followed by the reverse of a cross radial arc $[T_0, T_1]\to\mbr^2:t\mapsto c_{\rm inn}(t)=r_0(t)e^{i  t}$ where $0\leq  r_0\leq r_1$, followed by the reverse
of the radial path $[0,r_0(T_0)]\to\mbc:r\mapsto re^{i T_0}$.  We assume also that $r_0(t)<r_1(t)$ for $t\in (T_0,T_1)$.
The {\em interior} of this standard lasso is the bounded open subset of the plane whose boundary is formed by (part of) $c$.
 The {\em cone} of  $c$ is the set of all points $re^{i t}$ with
$\theta\in [T_0, T_1]$ and $0\leq r\leq c_{\rm out}(t)$. The path $c_{\rm out}$ is the {\em outer arc} of the lasso and $c_{\rm inn}$ is the {\em inner arc} of the lasso.

For the following result recall from (\ref{E:defMc}) and (\ref{UtMt}) the variables $M_c(t)$ and $u_c(t)$ associated to a cross radial path $c$.

\begin{prop}\label{P:u12props}
 Let $c_j:[T_0,T_1]\to \mbr^2:t\mapsto r_{j}(t) e^{it}$, for $j\in\{1,2\}$, be continuous paths, with $0\leq r_1\leq r_2$ and $0\leq T_0<T_1\leq 2\pi$. For $T_0      \leq T\leq  t\leq T_1  $,  let
\begin{equation}\label{E:defuj}
u_j(t) = u_{c_j}(t) u_{c_j}(T)^\ast
\end{equation}
and
\begin{equation}\label{E:defu12}
u_{12}(t) = u_1(t)^\ast u_2(t).
\end{equation}
Then
\begin{enumerate}
\item
$u_j(t)$ does not depend on the values of $c_j(s)$ for $s < T$.

\item  $u_{12}(t)$ is in the algebra $\mcb_{S}$,  where $S$ is the cone $\{be^{ia}: a\in [T,t], 0\leq b\leq r_2(a)\}$;

\item
$u_{12}(t)$ is freely independent from $\mcb_{S_{12}(t)^c}$, where $S_{12}(t)$ is the region
\begin{equation}\label{E:defS12}
S_{12}(t) = \set{be^{ia} : a \in [T, t], r_1(a) \leq b \leq r_2(a)},
\end{equation}
enclosed by the lasso specified by the inner arc $c_1|[T,t]$ and the outer arc $c_2|[T,t]$.
\item
The distribution of $u_{12}(t)$ is multiplicative semicircular with parameter given by the area of $S_{12}(t)$.
\end{enumerate}
\end{prop}

\begin{proof}
Let $u_x(t)$ be the free multiplicative Brownian motion generated by the free Brownian motion $x(t)$ on $[T_0,T_1]$; then according to Theorem~2 of \cite{BiaMultiplicativeBM}, $u_x(t) u_x(T)^\ast$ is precisely the multiplicative free Brownian motion generated by $x(t)-x(T)$ on $[T, T_1]$. Applying this with $x$ being $M_{c_j}$ we obtain part (1). This also implies that $u_j(t) \in \mcb_{S_j(t)}$, where $S_j(t)$ is the cone of the cross radial path $c_j|[T,t]$. Therefore $u_{12}(t) \in \mcb_{S_2(t)}$. Moreover, by Proposition~\ref{P:cocycle},
\begin{equation}\label{E:u12m21}
u_{12}(t) = u_{w}(t)
\end{equation}
where $w$ is the free Brownian motion $M_{c_2} - M_{c_1}$ twisted by $u_{c_1}$:
$$w=(M_{c_2} - M_{c_1})^{u_{{c_1}}},$$
and is freely independent from $\mcb_{S_{12}(t)}$. Parts (2) and (3) follow. Part (4)   follows from the remark made after  the statement of Proposition ~\ref{P:cocycle}.
\end{proof}

  It will be useful to introduce a special type of graph in the plane, specified by a set of radial and cross radial edges:

  \begin{definition}\label{D:grid} A  {\rm grid} $\Gamma$ is a graph whose vertices are points in $\mbr^2$ and whose oriented edges are radial and cross radial paths, along with their reverses, specified in more detail as follows. Let $0=\theta_0<\theta_1<\ldots<\theta_N = 2\pi$, and consider cross radial paths
$$[\theta_{j-1},\theta_j]\to\mbr^2: \theta\mapsto r_{jk}(\theta)e^{i\theta},$$
for $j\in \{1,\ldots, N \}$ and $k\in\{1,\ldots, m_j\}$, where
$$ r_{j\,0}(\theta)\stackrel{\rm def}{=}0\leq r_{j1}(\theta)  \leq\ldots\leq r_{j \,m_j}(\theta),\quad\hbox{for all $\theta\in [\theta_{j-1},\theta_j]$},$$
such that a pair of these arcs intersect  at most at  one endpoint. The edges of $\Gamma$ are the paths $r_{jk}$, their reverses $r_{jk}^{-1}$, and all the radial paths connecting endpoints of   cross radial edges. The vertices of $\Gamma$ are the endpoints of all the edges of $\Gamma$. The grid contains, in particular, contiguous radial segments along the ray $\theta=0$.
  \end{definition}

  By a cross radial edge of $\Gamma$ we mean one of the cross radial paths $r_{jk}$;
  a reverse cross radial edge is, naturally, the reverse of a cross radial edge.

The radial edges of  the grid $\Gamma$ described above  are the radial paths
$$[a(\theta_{j}), b(\theta_j)]: s\mapsto  se^{i\theta_j},$$
whose endpoints are endpoints of `successive' cross radial arcs proceeding outward along the ray $\theta=\theta_j$, fo $j\in \{1,\ldots, N\}$.

We order the cross radial arcs $r_{j\,k}$ in lexicographic order in $(j,k)$:
\begin{equation}\label{E:rij}
r_{1\,1}\leq  r_{1\,2}\leq\ldots \leq r_{1\,m_1}\leq r_{2\,1}\leq\ldots\leq r_{N \,m_{N }}.\end{equation}

 A {\em minimal} lasso in the grid $\Gamma$ is a path in $\Gamma$ consisting of a radial path followed by a cross radial arc $r_{j\,k}$,  with $k\in \{1,\ldots, m_j\}$, followed by a  reverse radial path, followed by $ {r }_{j\, k-1}^{-1}$  followed by a radial path back to the origin $o$.  Denote this minimal lasso by $l_{jk}$.

  Let $\underline{e}$ denote the initial point of an edge $e$, and $\overline{e}$ the final point.  If $v$ is a vertex of a grid $\Gamma$ let $[o,v]$ denote the radial path in $\Gamma$ from the origin $o$ to $v$, and $[v,o]=[0,v]^{-1}$.

If $e$ is a cross radial edge  in a grid $\Gamma$ then the loop $l_e=  [\overline{e}, o]\cdot e\cdot [o,\underline{e}]$ is clearly backtrack equivalent to the composition a sequence of minimal lassos  in $\Gamma$.  Specifically, if $e=r_{jk}$ then
$$\hbox{ $l_e$ is backtrack equivalent to $l_{ {j\,1}}\ldots l_{ {j\,k}}$.}$$
(See Figure  \ref{fig:loopdec} for an illustration.) Consequently, if $e$ is reverse cross radial then $l_e$ is the composite of a sequence of reversed minimal lassos.


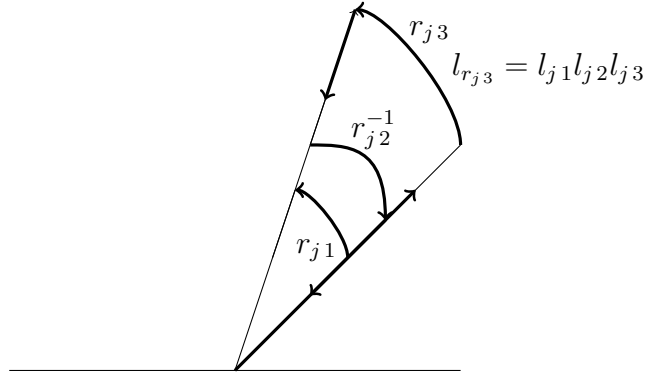
\begin{figure}\begin{center}
 \begin{tikzpicture}[scale=2]
\draw (-1.5,0) -- (1.5,0);

\draw [->,very thick](0,0) -- (1.2,1.2);

\draw  (0,0) -- (1.5,1.5);
\draw [<-,very thick] (1,1) .. controls (1,1.5) and (0.75,1.5) .. (0.5,1.5);

\draw [<-, very thick](0.6,1.8) -- (0.8, 2.4);

\draw [<-](0.8, 2.4) -- (0, 0);

\draw  (0.5,1.5) -- (0.8, 2.4);
\draw  [->,very thick] (1.5,1.5) .. controls (1.5,1.8) and (1,2.4) .. (0.8,2.4);

\draw  (0.25,.75) -- (0.4, 1.2);
\draw  [->,very thick] (.75,.75) .. controls (.75,.9) and (.5,1.2) .. (0.4,1.2);

\draw [->,very thick](1,1) -- (0.5,0.5);

\coordinate [label=right:$ r_{j\,2}^{-1}$] (cinn) at (.7,1.6);

\coordinate [label=left:$ r_{j\,3}$] (cout) at (1.5,2.25);

\coordinate [label=left:$ r_{j\,1}$] (clow) at (.75,.8);

\coordinate [label=left:{$l_{ r_{j\,3} }=l_{ j\,1 } l_{ j\,2 } l_{ j\,3 }$}] (decomp) at (2.8,2);

\end{tikzpicture}
\end{center}
\caption{Decomposing in terms of minimal lassos}

\label{fig:loopdec}
\end{figure}


\begin{prop}\label{P:loopsgrid}  Let $\Gamma$ be any grid.  Then:
\begin{itemize}
\item[(i)]  A loop  in $\Gamma$, based at a point $p$,  that consists only of radial or reverse radial edges is backtrack equivalent to the constant loop at $p$.

\item[(ii)] A path $c$ in $\Gamma$ that consists only of radial or reverse radial edges is backtrack equivalent to the radial/reverse-radial path, with no back tracking, from the initial point of $c$ to its final point.
\item[(iii)] The minimal lassos in $\Gamma$ can be listed in a sequence $l_1,\ldots, l_m$ such that the interior of $l_k$ is disjoint from the cone of $l_j$ whenever $j<k$. Any loop in $\Gamma$ based at $o$ is backtrack equivalent to the composite of a sequence of minimal lassos, or their reverses, in the grid. \end{itemize}
\end{prop}

Note that here we are treating the constant loop at $p$ as a loop in the graph $\Gamma$ and are viewing it as backtrack equivalent to $b^{-1}b$ for any edge $b$ initiating at $p$.

\begin{proof} (i) Let $c$ be a loop in $\Gamma$, based at $p$, and suppose $c=b_n\ldots b_1$, where $b_j$ is radial or reverse radial for all $j\in \{1,\ldots, n\}$.  Let $k\in\{1,\ldots, n\}$ be such that $\overline{b_k}$ is furthest from  $p$, with distance being measured by the minimal number of radial/reversed-radial edges needed to go from one point to the other. Then $k<n$ and the final point of $b_{k+1}$ must be the same as the initial point of $b_k$ (or else $\overline{b_{k+1}}$ would be further away from $p$ than $\overline{b_k}$). Since $b_k$ and $b_{k+1}$ are both radial or reverse radial it follows that $b_{k+1}=b_k^{-1}$.  By erasing $b_{k+1}b_k$ from $c$ we obtain a loop, still based at $p$, with fewer edges and that still satisfies the condition that every edge in it is either radial or cross radial. Inductively, $c$ is backtrack equivalent to the constant loop at $p$.

(ii) follows from (i) by considering the loop formed by the path $c$ followed by the `direct' path $d$ back from $\overline{c}$ to $\underline{c}$. Specifically, $d^{-1}c$ is backtrack equivalent to the constant path at $\underline{c}$, and so, composing on the left with $d$, it follows that $c$ is backtrack equivalent to $d$.

(iii)   Recall from (\ref{E:rij}) the sequence of cross radial edges  $r_{ij}$, and the corresponding minimal lassos $l_{ij}$ (as explained following (\ref{E:rij})) where the outer arc of $l_{ij}$ is given by $r_{ij}$. Now re-index the $(i,j)$ in lexicorgraphic order:
$$(1,1)< (1,2)<\ldots  < (1,m_1) < (2,1)<\ldots  <(N,m_N)$$
and let $l_k$ be $l_{ij}$ if $(i,j)$ is the $k$-th pair counted up, the lowest being $(1,1)$.

Let $l$ be a loop  in a grid $\Gamma$ based at the origin $o$ that is not entirely radial. Let $e$ be the first  edge of $l$ that is cross radial or reverse cross radial. Then $l$ is backtrack equivalent to the composite of  $l_e$ and a loop in $\Gamma$ that contains one fewer edge (than $l$) that is either cross   or reverse cross radial. Inductively, then $l$ is backtrack equivalent to the composite of a sequence of loops of the form $l_e$ and  a loop consisting solely of radial edges. The latter is backtrack erasable by (i), leaving $l$ equivalent to a composite of loops $l_e$, each of which is, as noted before,    backtrack equivalent to the composite of   minimal lassos
 and their reverses.  \end{proof}

For an elegant study of loops in graphs on surfaces, see L\'evy \cite{Le10}. We note there some interesting facts: each backtrack equivalence class contains a loop of minimal edge length, and these form a group under composition of loops followed by backtrack reduction to minimal edge length.

We have now the following immediate consequence of Proposition \ref{P:u12props}:

\begin{theorem}\label{T:fregrid} Let $\Gamma$ be a grid in $\mbr^2$, and let $l_1,\ldots, l_m$ be the basic minimal lassos of $\Gamma$  listed in sequence so that
the interior of $l_k$ is disjoint from the cone of $l_j$ for all $j<k$. Then $u_{l_1},\ldots, u_{l_m}$ are  freely independent, and each $u_{l_j}$ is multiplicative semicircular with parameter $|S_j|$,  where $S_j$ is the interior of the lasso $l_j$.
\end{theorem}

 \section{Review of $U(N)$ planar Yang-Mills}\label{s:revN}

 Everything we have done so far works  even if the variables $b_f$ and $u_c$
 are random variables on some probability space, with $b_f$ taking values in the space of all
 skew-hermitian $N\times N$ matrices, and $u_c$ taking values in $U(N)$. The equation (\ref{E:ux}) for parallel transport is then taken as an
 It\^o stochastic differential equation.  Free independence is replaced by classical probabilistic independence.
This is the  quantum Yang-Mills theory on the plane with
gauge group $U(N)$ as developed by Gross et al. \cite{GKS89} and Driver \cite{Dr89}. The formal expression for the quantum $U(N)$ Yang-Mills
measure  is
\begin{equation}\label{e:formalym2un}
\frac{1}{Z_g}e^{-\frac{1}{2g^2} S_{\rm YM}(A)}\,[DA],\end{equation}
where $S_{\rm YM}(A) =\frac{1}{2}|\!|F^A|\!|_{L^2}^2$, with $F^A$ being curvature of a connection $A$, is the Yang-Mills action, and $g>0$ a (coupling) constant.  Our objective is the limit of this measure as $N\to\infty$ but holding $g^2N$ fixed, and so we set $g^2=1/N$. We now review very briefly the basic ideas and results for $U(N)$ planar quantum Yang-Mills theory (for more see    \cite{Dr89} and \cite{GKS89}).

The Lie algebra $u(N)$ consists of all skew-hermitian $N\times N$ matrices, and
\begin{equation}\label{E:ipuN}
\la H, K\ra=-\Tr(HK)\end{equation}
specifies an inner product on $u(N)$ that is invariant under the conjugation action of $U(N)$ on $u(N)$. Let $E_1,\ldots, E_{N^2}$ be an orthonormal basis of $u(N)$. The Gaussian measure (\ref{e:formalym2un}), with $g^2=1/N$, is modeled rigorously as follows. There is a probability space $(\Omega_N, {\mathcal F}_N,\mbp_N)$   and a linear map
$$L^2(\mbr^2)\otimes u(N)\to L^2(\mbp_N):\phi\mapsto \phi^\sim$$
such that $\phi^\sim$ is Gaussian of mean $0$ and variance $\frac{1}{N}|\!|\phi|\!|^2$.  For any
$f\in L^2(\mbr^2)$ and every basis element $E_j$ there is thus a Gaussian random variable $(fE_j)^\sim$, depending linearly on $f$, of mean $0$ and variance $\frac{1}{N}|\!|f|\!|_{L^2(\mbr^2)} $; thus the random variable
$$ \sum_{j=1}^{N^2}(fE_j)^\sim E_j,$$
 with values in the matrix Lie algebra $u(N)$,   depends linearly on $f$ and  the components $(fE_j)^\sim $
are independent Gaussians with mean $0$ and variance $\frac{1}{N}|\!|f|\!|_{L^2(\mbr^2)}$.  Thus
\begin{equation}\label{E:defbNf}
b_{N,f}=  \sum_{j=1}^{N^2}(fE_j)^\sim iE_j\end{equation}
is a hermitian random matrix  such that $\Tr(b_{N,f}iE_j)$ is Gaussian with mean $0$ and variance  $\frac{1}{N}|\!|f|\!|_{L^2(\mbr^2)}$.

Now consider a cross radial path $c:[T_0,T_1]\to\mbr^2:\theta\mapsto r_c(\theta)e^{i\theta}$. Let
\begin{equation}\label{E:MNct}
M_{N,c}(t)= b_{N,S_c(t)},\end{equation}
where $S_c(t)$ is the usual cone for $c|[T_0,t]$ seen in (\ref{E:defSct2}). Then
$$t\mapsto  iM_{N,c}(t)$$
 is a $u(N)$-valued Brownian motion clocked by the scaled area function $t\mapsto \frac{1}{N}|S_c(t)|$, where $|A|$ denotes Lebesue measure of $A\subset\mbr^2$.  Stochastic parallel transport
 $$t\mapsto h_c(t),$$
 with $t\in [T_0,T_1]$ and each $h_c(t)$ a $U(N)$-valued random variable, solves
 the It\^o stochastic differential equation
 \begin{equation}\label{htMtN}dh_c = i \bigl(dM_{N,c} \bigr) h_c    -\frac{1}{2}\bigl(dM_{N,c} \bigr)^2h_c,\qquad\hbox{with}\quad h_c(T_0)=I.\end{equation}
 As with $u_c$ in (\ref{E:defucT1}), we allow ourselves the notational ambiguity of denoting parallel transport along the full path $c$ also by $h_c$:
 \begin{equation}\label{E:defhcT1}
 h_c=h_c(T_1).\end{equation}
The variable $h_c$ has values in $U(N)$ and has density (with respect to unit mass Haar measure) given by
\begin{equation}\label{E:denhct}
Q_{\frac{1}{N} |S_c|}(x),\end{equation}
where $Q_t(x)$ is the heat kernel on $U(N)$ and $ |S_c|$ is the area of the cone $S_c=\{re^{it}:t\in [T_0, T_1], 0\leq r\leq r_c(t)\}$.  (To avoid notational clutter, we are not indexing $h_c$ by $N$. ) The heat kernel $Q_t(x)$ is the solution of the heat equation
\begin{equation}\label{E:heat}
\frac{\partial Q_t(x)}{\partial t}=\frac{1}{2}\Delta_{U(N)}Q_t(x),\end{equation}
where $\Delta_{U(N)}$ is the Laplacian on $U(N)$, and satisfies the initial condition
$$\lim_{t\downarrow 0}\int_{U(N)}f(x)Q_t(x)\,dx=f(I),$$
for continuous functions $f$ on $G$, where $dx$ is unit mass Haar measure on $U(N)$.
Probabilistically, $Q_t(x)$ is the density of the time-$t$ position of standard Brownian motion on $U(N)$.

For a basic curve $c$,    define $h_c$ exactly analogously to $u_c$ as in Definition \ref{d:hol}:

\begin{definition}\label{d:stochhol}   For a basic curve $c:[T_0 T_1]\to\mbr^2$ we define
the {\em stochastic parallel transport} along $c$ to be
\begin{equation}\label{E:defhol2}
h_c=h_{c_m}\cdots h_{c_1}\end{equation}
if $c$ is the composite
$$c=c_m\cdots c_1,$$
where each $c_j$ or its reverse is radial or cross radial, and $h_{c_j}=1$ if $c_j$ is radial.  If $c$ is a loop, then we call $h_c$ the {\em stochastic holonomy} around $c$. \end{definition}

The following result is the exact analog of  our free Theorem \ref{T:fregrid}:

\begin{theorem}\label{T:unholgrid} Let $\Gamma$ be a grid in $\mbr^2$, and let $l_1,\ldots, l_m$ be the basic minimal lassos of $\Gamma$  listed in sequence so that
the interior of $l_k$ is disjoint from the cone of $l_j$ for all $j<k$. Then the variables $h_{l_1},\ldots, h_{l_m}$ are independent, and each $h_{l_j}$ has density on $U(N)$, relative to unit mass Haar measure, given by $Q_{\frac{1}{N}|S_j|}(x)$  where $S_j$ is the interior of the lasso $l_j$.
\end{theorem}

A special feature of the finite-$N$ theory is that it can be reformulated as a lattice gauge theory.  Let $\mbe_{\Gamma}$ be the set of oriented edges of $\Gamma$, $\mbv_{\Gamma}$ the set of vertices, and $\mbf_{\Gamma}$ the set of faces (bounded components of the complement in $\mbr^2$ of the set of points on all the  edges of $\Gamma$). We equip each face $F$ with the standard orientation from $\mbr^2$ and also equip it with a fixed, arbitrarily chosen, basepoint on $\partial F$. The set
\begin{equation}\label{E:defAGamma}
{\mca}_{\Gamma}=\{x:\mbe_{\Gamma}\to U(N)\,:\, x(e^{-1})=x(e)^{-1}\quad\hbox{for all $e\in \mbe_{\Gamma}$ }\}\end{equation}
is a discrete version of the space of all connections restricted to $\Gamma$.  The group of gauge transformations is then
\begin{equation}\label{E:defGGam}
{\mathcal G}_{\Gamma}= U(N)^{\mbv_{\Gamma}},\end{equation}
with any $\theta\in {\mathcal G}_{\Gamma}$ acting on $\mca_{\Gamma}$ by
\begin{equation}\label{E:gaugetr}
x^{\theta}(e)=\theta(\overline{e})^{-1}x(e)\theta(\underline{e})\quad\hbox{for all edges $e$.}\end{equation}
The {\em discrete Yang-Mills measure} $\mu_{\Gamma}$ is the probability measure on $\mca_{\Gamma}$ given by
\begin{equation}\label{E:dym}
d\mu_{\Gamma}(x)=\prod_{F\in\mbf_{\Gamma}}Q_{\frac{1}{N}|F|}\bigl(x(\partial F)\bigr),\end{equation}
where $|F|$ is the area of the face $F$, and $x(\partial F)=x(b_k)\ldots x(b_1)$ if
$b_,\ldots, b_k$, composed in that order, form the boundary $\partial F$, starting with the given basepoint on $\partial F$. This measure is clearly invariant under the action of ${\mathcal G}_{\Gamma}$, and can be viewed also as a measure on the quotient space ${\mca}_{\Gamma}/{\mathcal G}_{\Gamma}$.

 With these structures in place, we have:

\begin{theorem}\label{T:unholcigrid} Let $\Gamma$ be a grid in $\mbr^2$, and let $c_1,\ldots, c_n$ be loops, based at $o$, in $\Gamma$. Then
\begin{equation}\label{E:ymexpec}
\int f(h_{c_1},\ldots, h_{c_n})\,d\mbp_N = \int_{\mca_{\Gamma}}f\left(x(c_1),\ldots, x(c_n)\right)\,d\mu_{\Gamma}(x)
\end{equation}
for all bounded measurable functions $f$ on $U(N)^n$.
\end{theorem}
For a proof we refer to   \cite[Theorem 6.4]{Dr89} (with axial gauge fixing, rather than radial gauge fixing) and \cite[Theorem 8.4]{Se7} (a very general case for surfaces).
Technically, Theorem \ref{T:unholcigrid} is a consequence of Theorem \ref{T:unholgrid},  the essence of the idea being that the original gauge fixing can be undone at the lattice level to produce the gauge invariant expression (\ref{E:ymexpec}).

The convolution property
$$\int_{U(N)}Q_t(xy)Q_s(y^{-1}z)\,dy=Q_{t+s}(xz)$$
of the heat kernel makes it possible to combine adjacent faces $F_1$ and $F_2$, sharing a common edge $e$ as follows:
\begin{equation}\label{E:Qconv}
\int_{U(N)}Q_{\frac{1}{N}|F_1|}\bigl(x(\partial F_1)\bigr)Q_{\frac{1}{N}|F_2|}\bigl(x(\partial F_2)\bigr)\,dx(e)=Q_{\frac{1}{N}|F|}\bigl(x(\partial F)\bigr),\end{equation}
where $F$ is formed by coalescing $F_1$ and $F_2$ and deleting $e$.  This can be used on the right side of (\ref{E:ymexpec}) for  every edge $e$ that does not appear in any of the loops $c_i$, thereby eliminating all such edges from the integration. What is left is
an integration over  the `gauge fields' $x$  defined on the graph specified by the intersection points of the loops $c_i$ and the segments of the $c_i$ running between these intersection points.  Consider the   case where $\Gamma$ provides a triangulation of the unit disk $D$ and $l$ is a loop of the form $[p,o]\cdot\partial D\cdot [o,p]$, where $p$ is on the boundary of $D$ and $\partial D$ traces out this boundary in the positive sense, starting at $p$. Then, eliminating all edges not on $\partial D$, we have
\begin{equation}\label{E:fhlD}
\int f(h_l)\,d\mbp_N=\int_{U(N)}f(y)\,Q_{\frac{1}{N}|D|}(y)dy,\end{equation}
for all bounded measurable central functions $f$ on $U(N)$ (centrality allows us to `cancel off' the edges on $[o,p]$ against those on $[p,o]$). Next, suppose $l$ is of the form $[o,p]\cdot c\cdot [p,o]$, where $c$ is a basic simple closed curve based at $p$ enclosing a region $S$. By drawing sufficiently many radial and cross radial arcs we obtain a grid part of which provides a triangulation $T$ of $S$. There is a  (Jordan-Sch\"onflies) homeomorphism of $S$ onto the unit disk and, by further results from topology (the two-dimensional Hauptvermutung) , there is a homeomorphism which carries a subdivision of $T$ to a subdivision of the triangulation of $D$ specified by, say, two radial segments along with two arcs on $\partial D$. Putting all this together with (\ref{E:fhlD}) it follows that
\begin{equation}\label{E:fholjord}
\int f(h_l)\,d\mbp_N=\int_{U(N)}f(y)\,Q_{\frac{1}{N}|S|}(y)dy\end{equation}
for all bounded measurable central functions $f$ on $U(N)$. Aside from the proof, this identity is easy to verify in most examples using Theorem \ref{T:unholcigrid} and the convolution property (\ref{E:Qconv}).

Here is a particularly interesting consequence of (\ref{E:fholjord}) (see
\cite[eq. (5.13)]{Sen08a}):

\begin{equation}\label{E:fholjord2}
\lim_{N\to\infty}\int {\rm tr}_N(h_l^k)\,d\mbp_N=e^{- \frac{k}{2}|S|}P_k(|S|),\end{equation}
where $k$ is any non-negative integer, $P_k$ is the polynomial defined in (\ref{diffeqPk}),  and $\tr_N=\frac{1}{N}\tr$ is the normalized trace on $N\times N$ matrices (for negative $k$ simply note that $h_l^{-1}$ has the same distribution as $h_l$). Comparing with (\ref{E:distribux}) we observe that
\begin{equation}\label{E:tauhol}
\tau(u_l^k)=\lim_{N\to\infty}\tr_N(h_l^k),\end{equation}
for all positive integers $k$,
at least when $l$ is a standard lasso.

There is another important feature of $U(N)$ planar quantum Yang-Mills theory that we need to note. {\em If $c_1,\ldots, c_m$ are basic simple closed loops, based at $o$, with enclosing disjoint regions  then $h_{c_1},\ldots, h_{c_m}$ are independent variables.}
This too follows from  (\ref{E:ymexpec})  by eliminating all edges that do not lie on the loops $c_i$.

As the discussion above indicates, the loop expectation values (\ref{E:ymexpec}) depend only on the areas of the bounded regions in the complement of the paths $c_i$ and the topologies of these regions.  For a general theory of such topological probability field theories see L\'evy \cite{Le10} (earlier works include \cite{Le03, Se3, Se7}).

 \section{The Large-$N$ Limit and Free Independence of Holonomies}\label{s:lnl}

In this section we show that the large-$N$ limit of the planar $U(N)$ quantum Yang-Mills theory is indeed the free theory.  For finite $N$, we consider matrix-valued random variables $A$ for which we use the `trace functional'
\begin{equation}\label{E:EAN}
\tau_N(A)=\mbe[\tr_N(A)].\end{equation}

 For a basic loop $c$ in $\mbr^2$,  let $h_c$ be the $U(N)$-valued stochastic holonomy  given by (\ref{E:defhcT1}) and by the description before Theorem \ref{T:unholgrid}, and let $u_c$ be its free analog as in Definition \ref{d:hol}.

For a basic lasso $l$ the variable $h_l$ has density $Q_{\frac{1}{N}|S|}(x)$,  where $S$ is the interior of the lasso (region enclosed by the lasso head).  Wigner's semicircle law \cite{Wig55,Wig58} implies that
$$h_l\to u_l,\quad\hbox{as $N\to\infty$}$$
in distribution (in fact we have already observed this in (\ref{E:fholjord2})). Voi\-cu\-les\-cu's theorem \cite{Vo91} in this context implies that if $l_1,\ldots, l_n$ are loops for which $h_{l_1},\ldots, h_{l_n}$ are independent
as $U(N)$-valued random variables, for each $N$,   then $ u_{l_1},\ldots, u_{l_n} $ are freely independent and
\begin{equation}\label{E:hcucn}
(h_{l_1}, h_{l_1}^\ast,\ldots, h_{l_n}, h_{l_n}^\ast)\to (u_{l_1}, u_{l_1}^\ast,\ldots, u_{l_n}, u_{l_n}^\ast),\quad\hbox{as $N\to\infty$,}\end{equation}
in distribution.

 \begin{theorem}\label{T:largeN} Let $c_1\ldots, c_n$ be basic loops, all based at $o$, with finitely many mutual intersection points.  Then
\begin{equation}\label{E:hcucn2}(h_{c_1}, h_{c_1}^\ast,\ldots,h_{c_n}, h_{c_1}^\ast)\to (u_{c_1}, u_{c_1}^\ast,\ldots,u_{c_n}, u_{c_n}^\ast),\quad\hbox{as $N\to\infty$,}\end{equation}
in distribution, where $u_c$ denotes the free   holonomy around a loop $c$ and $h_c$  the
$U(N)$-valued stochastic holonomy.
 \end{theorem}
 \begin{proof} By subdividing cross radial paths appearing in the $c_j$, as necessary, we can assume that distinct cross radial pieces intersect
 at most at one or both endpoints.  Draw radial paths from the origin to the endpoints of   cross radial pieces of all the loops $c_j$, as well as the initial ray $\theta=0$.   These constructions produce a grid $\Gamma$.  By Proposition \ref{P:loopsgrid} there is
 a sequence of standard lassos $l_{1},\ldots,l_M$, with disjoint interiors, such that each $c_j$ is backtrack equivalent to a composite of these lassos and their
 reverses.
 From the observations made above for (\ref{E:hcucn}), we have
 $$(h_{l_1},h_{l_1}^\ast, \ldots,h_{l_M}, h_{l_M}^\ast)\to (u_{l_1},u_{l_1}^\ast\ldots,u_{l_M}, u_{l_M}^\ast) \quad\hbox{as $N\to\infty$}.$$
 Since each $h_{c_j}$ is a product of some of the $h_{l_k}$ and their inverses, and $u_{c_j}$ is the product of the corresponding unitaries $u_{l_k}$ and their inverses, the limit in the
 relation (\ref{E:hcucn2}) follows.
 \end{proof}

\begin{theorem}\label{T:holosloop} Suppose $l$ is a basic    closed curve, comprised of an initial radial ray from $o$ to a point $p$, followed
by a simple closed curve $c$ based at $p$, followed by the radial path back to $o$.
Then $u_{l}$ has multiplicative semicircular with parameter given by the area enclosed by the loop $c$.
\end{theorem}
\begin{proof} Draw radial paths from the origin $o$ to the initial and final points of all cross radial paths which appear in the loop $c$.  Add more radial paths and cross radial arcs to form a grid.  Let $l_1,\ldots, l_m$ be the minimal lassos of the grid listed in sequence so that  the interior of $l_k$ is disjoint from the cone of $l_j$ is $j<k$. Then, by Proposition
\ref{P:loopsgrid}, the loop $l$ is    backtrack equivalent to a composite of some of the $l_j$ and their reverses. Denoting by $h_a$ the
$U(N)$-valued   stochastic holonomy around a basic loop $a$, then by Theorem \ref{T:largeN} we have
$$(h_{l_1},h_{l_1}^\ast,\dots,h_{l_m}, h_{l_m}^\ast)\to (u_{l_1},u_{l_1}^\ast,\ldots, u_{l_m}, u_{l_m}^\ast)$$
in distribution,
as $N\to\infty$. Hence, writing $l$ as a composite of the $l_j$ and their reverses, it follows that $h_l\to u_l$ in distribution, as $N\to\infty$.
But from (\ref{E:fholjord}) of   the $U(N)$ theory, $h_l$ has density $Q_{S/N}$ on $U(N)$, where $S$ is the area enclosed by the loop $c$. Hence,
the free limit $u_l$ is semicircular with parameter $S$.
\end{proof}

\begin{theorem}\label{T:holosloopsfree}  Suppose $c_1,\ldots, c_n$ are basic simple closed loops in $\mbr^2$, all based at the origin $o$, with disjoint interiors. Then
$u_{c_1},\ldots, u_{c_m}$ are freely independent.
\end{theorem}
\begin{proof}   Form a grid $\Gamma$ by using all cross radial segments of the paths $c_i$ and all radial lines running from $o$ to the endpoints of   these cross radial segments, as well as the initial ray $\theta=0$.  By Proposition \ref{P:loopsgrid}   there are standard lassos $l_1,\ldots, l_m$   with disjoint interiors, listed/indexed so that the interior of $l_k$
is disjoint from the cone of $l_j$ for $j<k$, and such that each $c_i$ is backtrack equivalent to a composite of the $l_j$ and their reverses.
Then the $U(N)$   stochastic holonomies $h_{c_i}$ converge jointly, in distribution,   to the corresponding
$u_{c_i}$. Now from $U(N)$-valued Yang-Mills theory, the variables $h_{c_1},\ldots, h_{c_n}$ are independent, and so their limits $u_{c_1},\ldots, u_{c_n}$ are
freely independent.
\end{proof}

\section*{Acknowledgments}  ANS also thanks Thierry L\'evy
and Kalyan B. Sinha
for useful discussions.


 \bibliographystyle{amsplain}


\def\cprime{$'$}

\end{document}